\documentclass[11pt]{amsart}

\usepackage[utf8]{inputenc}
\usepackage[T1]{fontenc}
\usepackage{amssymb,stix,calrsfs}
\usepackage{xcolor}

\newtheorem{mthm}{Theorem}

\newtheorem{thm}{Theorem}

\newtheorem{prop}[thm]{Proposition}
\newtheorem{ques}[thm]{Question}

\theoremstyle{definition}
\newtheorem{defn}[thm]{Definition}

\theoremstyle{remark}
\newtheorem{rem}[thm]{Remark}
\newtheorem{exa}[thm]{Example}

\DeclareMathOperator{\im}{im}

\DeclareMathOperator{\Ric}{Ric}

\DeclareMathOperator{\Scal}{Scal}
\DeclareMathOperator{\Sym}{Sym}
\DeclareMathOperator{\tr}{tr}

\begin{document}

\title{Instability of conformally Kähler, Einstein metrics}

\author{Olivier Biquard \and Tristan Ozuch}
\address{Sorbonne Université and Université Paris Cité, CNRS, IMJ-PRG, F-75005 Paris, France}
\email{olivier.biquard@sorbonne-universite.fr}
\address{MIT Department of Mathematics, 182 Memorial Dr, Cambridge, MA 02139, USA}
\email{ozuch@mit.edu}

\begin{abstract}
    We prove the instability of conformally Kähler, compact or ALF Einstein $4$-manifolds with nonnegative scalar curvature which are not half conformally flat. This applies to all the known examples of gravitational instantons which are not hyperKähler and to the Chen-Lebrun-Weber metric in particular. 
\end{abstract}

\maketitle

\section*{Introduction}

Einstein metrics are considered \textit{optimal} in Riemannian geometry, especially in dimension $4$. Their \textit{stability} in physics or as fixed points of Ricci flow is a crucial question: one should be able to go away from unstable metrics by applying Ricci flow. In both contexts, the \textit{stability} is defined as being a local maximum of the Hilbert-Einstein functional $\mathcal{S}:g\mapsto \int_M \Scal^gdv_g$ among Yamabe metrics, or of Perelman's functionals among general metrics--these conditions are equivalent to second perturbative order at Einstein metrics.

For Ricci-flat metrics, special holonomy implies stability, see for instance \cite{DaiWanWei05}. An open question is that of the converse, which is sometimes formulated as a conjecture, see \cite{Ach20} under more assumptions.
\begin{ques}[Folklore]
    Are Ricci-flat metrics with \emph{generic} holonomy \emph{unstable}?
\end{ques}
Stability is a subtle question which has mostly been studied on the most symmetric Einstein metrics, and instabilities have been detected numerically on some of the simplest Einstein metrics, \cite{You83a,You83b}. 

In this article, we tackle the question for essentially all known examples of compact or complete Einstein $4$-manifolds with nonnegative scalar curvature at once. This is because, all known examples of Einstein metrics with nonnegative scalar curvature and curvature decay at infinity have a specific structure: they are \textit{conformally Kähler}, or \textit{Hermitian}. This condition is also called \textit{type $D^+$} in the physics literature by analogy with the Petrov classification in Lorentzian geometry.

\begin{mthm}
  Let $(M^4,g)$ be an Einstein manifold which is conformal to a Kähler metric. Suppose
  \begin{itemize}
  \item $(M,g)$ is compact with positive scalar curvature, or
  \item $(M,g)$ is one of the known examples of ALF gravitational instantons (see Section \ref{sec:instability-alf}).
  \end{itemize}
  Then if $(M^4,g)$ is not half-conformally flat, then it is \emph{unstable}.
\end{mthm}
Note that the method of proof is abstract and uniform: destabilizing symmetric 2-tensors are produced as compositions $\omega_- \, \circ \, \tau$, where $\omega_-$ is a harmonic anti-selfdual 2-form and $\tau$ is a Killing selfdual 2-form, see Section \ref{sec:instability-compact}.

Except for $\mathbb{S}^4$ and $\mathbb{CP}^2$, the theorem applies to all the known Einstein metrics with positive scalar curvature, i.e. the metrics of Page, Chen-LeBrun-Weber and the Kähler-Einstein Fano manifolds; the new case here is the Chen-LeBrun-Weber metric. It also applies to all of the known so-called gravitational instantons which are not hyperKähler: the Riemannian Kerr (including Schwarzschild), Taub-Bolt and Chen-Teo metrics. This is likely to be the full list, see \cite{AADNS,LiM23b} for the ALF case and \cite{LiM23a} for the ALE case.

Delay \cite{Del24} was the first to realize that Einstein stability can be checked in the conformal class. Elaborating on this idea, our proof relies on three main ingredients: a Weitzenböck formula from \cite{BiqRol09}, the conformal invariance of an operator closely related to the second variation of $\mathcal{S}$, and the construction of specific harmonic anti-selfdual $2$-forms on conformally Kähler Ricci-flat manifolds.

We treat the conformal invariance in Section \ref{sec:conformal-invariance}, with an approach which is different from the one by Delay. We then prove the instability result in the compact case in Section \ref{sec:instability-compact}. We finally deal with the ALF setting in Section \ref{sec:instability-alf}.

\section{Conformal invariance}
\label{sec:conformal-invariance}

Let $(M^4,g)$ be an oriented Riemannian manifold. Suppose $g$ is Einstein: $\Ric_g=\Lambda g$. From \cite[(1.180a)]{Bes87}, the linearization of the Einstein operator $\Ric-\Lambda$, acting on symmetric 2-tensors, is given by the formula
\begin{equation}
  \label{eq:1}
  d(\Ric-\Lambda) = \frac12 \nabla^*\nabla - \operatorname{Rm} - \delta^* B,
\end{equation}
where $B=\delta+\frac12 d \tr$ is the Bianchi operator, $\delta^*$ is the adjoint of the divergence, and $\operatorname{Rm}$ is the natural action of the Riemannian curvature on symmetric 2-tensors. The last term is a gauge term, so one considers often the linearization with gauge fixing
\begin{equation}
  \label{eq:2}
  L := \frac12 \nabla^*\nabla - \operatorname{Rm}.
\end{equation}
{The diffeomorphism invariance and the Bianchi identity $B(\Ric-\Lambda)=0$ imply the identities, for sections of $\Sym^2_0 T^*M $:
  \begin{equation}
    \label{eq:29}
    L\delta_0^* = \delta_0^*(B\delta^*), \quad \delta L = (B\delta^*)\delta,
  \end{equation}
where $\delta_0^*$ is the trace-free part of $\delta^*$.}
In dimension four we have the bundles of selfdual or anti-selfdual 2-forms $\Lambda_\pm$ (we will denote by $\Omega_\pm$ the corresponding space of sections), and a special isomorphism $\Lambda_- \otimes \Lambda_+=\Sym^2_0 T^*M $, which is given by
\begin{equation}
  \label{eq:3}
  \alpha_-\otimes \alpha_+ \longmapsto \alpha_-\circ\alpha_+,
\end{equation}
identified with the composition of the anti-symmetric endomorphisms associated to the $2$-forms $\alpha_\pm$. We will also denote by $\Gamma(E)$ the space of sections of the bundle $E$. Observe that by \cite[(4.6)]{BiqRol09}, for an Einstein metric we have on trace-free symmetric 2-tensors:
\begin{equation}
  \label{eq:4}
  L = d_-d_-^* - W_+ -\frac{\Scal}{12},
\end{equation}
where $d_-:\Gamma(\Lambda^1\otimes\Lambda_+)\rightarrow \Gamma(\Lambda_-\otimes \Lambda_+)$ is the composition of the covariant exterior derivative on 1-forms with values in $\Lambda_+$ and the projection  $\Lambda^2\otimes\Lambda_+\to\Lambda_-\otimes\Lambda_+$ on the anti-selfdual part. The action of $W_+$ is given by $\alpha_-\circ\alpha_+\mapsto \alpha_-\circ W_+(\alpha_+)$.

Let us use the spinor notation: the basic spinor bundles are $S_\pm$ (they exist at least locally), and their symmetric powers are $S_\pm^k$. Then $TM=S_+\otimes S_-$ and $\Lambda_\pm=S_\pm^2$.

For simplicity of notation we will now drop the tensor product signs. We have the decomposition
\begin{equation}
  \label{eq:5}
  \Lambda^1 \Lambda_+=(S_+ S_-)S_+^2=S_+^3S_- \oplus S_+S_-.
\end{equation}
We can identify the second summand with $\Lambda^3$ and we have the projection
\begin{equation}
  \label{eq:6}
  \pi:\Lambda^1 \Lambda_+ \rightarrow \Lambda^3, \qquad \pi(\alpha\otimes \omega)=\alpha\wedge \omega.
\end{equation}
Conversely, we have the embedding
\begin{equation}
  \label{eq:7}
  \sigma:\Lambda^3 \hookrightarrow \Lambda^1\Lambda_+, \qquad \sigma(\varphi)=\sum_1^4 e^i\otimes(e_i \lrcorner \varphi)_+,
\end{equation}
where $(e_i)$ is an orthonormal basis of $TM$. One calculates
\begin{equation}
  \label{eq:8}
  \sigma^*=\pi, \qquad \pi\circ\sigma = \frac32.
\end{equation}
We justify the second equality: the operator $\pi\circ\sigma$ is a $SO(4)$-invariant operator on the irreducible representation $\Lambda^3$, and therefore is a constant. We have the easy equality $\sum_1^4 e^i\wedge(e_i\lrcorner \varphi)=3\varphi$ for any 3-form $\varphi$. On the other hand,
\[ \sum_1^4 e^i\wedge(e_i\lrcorner \varphi)_+ = \sum_1^4 e^i\wedge(e_i\lrcorner \varphi)_- \]
because the two sides are exchanged by a change of orientation. Therefore
\[ \pi\circ\sigma(\varphi) = \sum_1^4 e^i\wedge(e_i\lrcorner \varphi)_+ = \frac12 \sum_1^4 e^i\wedge(e_i\lrcorner \varphi)=\frac32 \varphi. \]

We now decompose $d_-^*$ on the sum (\ref{eq:5}): we call $T$ the projection on the first summand, and thanks to the identification (\ref{eq:3}) we can consider the divergence on traceless 2-tensors as an operator $\delta:\Gamma(\Lambda_-\Lambda_+)\rightarrow \Omega^1$. 
\begin{prop}
  We have the formulas:
  \begin{align}
    \label{eq:9}
    d_-^* &= T - \frac23 \sigma * \delta , \\ \intertext{ and}
    d_-d_-^* &= T^*T + \frac23 \delta_0^*\delta.\label{eq:10}
  \end{align}
\end{prop}
\begin{proof}
  For the first formula, by definition of $T$ we are left with proving $\pi \circ d_-^*=-\frac23 \pi \sigma * \delta$, that is $\pi \circ d_-^*=- * \delta$ by (\ref{eq:8}), or $\delta=* \pi \circ d_-^*$. Any section of $\Lambda_-\Lambda_+$ decomposes as a sum of decomposed sections $\alpha_-\alpha_+$, and to calculate at a point $p$ we can suppose that $(\nabla\alpha_+)(p)=0$. Then at the point $p$ we have
\begin{align*}
    \delta(\alpha_-\alpha_+) &= -\sum_1^4 e_i\lrcorner (\nabla_{e_i}\alpha_- \alpha_+)\\
              &=-\sum_1^4 \alpha_+(\nabla_{e_i}\alpha_-(e_i)) \\
              &=-\sum_1^4 \alpha_+(e_i\lrcorner \nabla_{e_i}\alpha_-) \\
              &=-\sum_1^4 * \big( (e_i\lrcorner \nabla_{e_i}\alpha_-) \wedge \alpha_+ \big) \\
    &= * \pi d_-^*(\alpha_-\alpha_+).
  \end{align*}
  Here we used the identities $\alpha_\pm(X)=X\lrcorner \alpha=\pm*(X\wedge \alpha_\pm)$ for the $2$-form $\alpha_\pm$ also seen as an antisymmetric endomorphism of $TM$ and a tangent vector $X$ (also identified to a 1-form via the metric).

  The second formula is a consequence of the first formula and of (\ref{eq:8}). 
\end{proof}

It follows that the operator $L$ from (\ref{eq:4}) can be rewritten as
\begin{equation}
  \label{eq:11}
  L = T^*T - W_+ + \frac23 \delta_0^*\delta - \frac{\Scal}{12} = P + \frac23 \delta_0^*\delta - \frac{\Scal}{12},
\end{equation}
where we define the operator $P$ on $\Sym_0^2T^*M$ by
\begin{equation}
  \label{eq:12}
  P = T^*T - W_+.
\end{equation}

The interest of the operator $P$ comes from:
\begin{prop}\label{prop: conformal covariance}
  The operator $P$ acting on $\Sym^2_0T^*M$ has the conformal invariance
  \begin{equation}
    \label{eq:13}
    P^{f^2g}h = f^{-1} P^g(f^{-1}h).
  \end{equation}
\end{prop}
\begin{proof}
  The operator $T:S_+^2S_-^2 \rightarrow S_+^3S_-$ is a gradient operator, and it is known that such operators are always conformally invariant \cite{Feg76}. In our case we can apply the formula \cite[(3.22)]{Pil11b} to obtain
  \begin{equation}
    \label{eq:14}
    T^{f^2g} = \phi \circ f^{-2} \circ T^g \circ f \circ \phi^{-1},
  \end{equation}
  where the identification $\phi$ between bundles for $g$ and $f^2g$ is chosen to be isometric, and we identify the function $f$ with the multiplication by $f$. It follows that:
  \begin{align}
    \label{eq:15}
    (T^*)^{f^2g} &= \phi \circ f^{-3} \circ (T^*)^g \circ f \circ \phi^{-1}, \\
    (T^*T)^{f^2g} &= \phi \circ f^{-3} \circ (T^*T)^g \circ f \circ \phi^{-1}. \label{eq:16}
  \end{align}

  For symmetric 2-tensors, the abstract identification $\phi$ between bundles for $g$ and $f^2g$ is the isometric map $\phi(h)= f^2h$, so in terms of a fixed section of $h$ of $\Sym^2_0 T^*M$ (independent of the metric), the formula (\ref{eq:16}) reads
  \begin{equation}
    \label{eq:17}
    (T^*T)^{f^2g} h = f^{-1} (T^*T)^g (f^{-1}h).
  \end{equation}

  On the other hand, for the Weyl tensor seen as an endomorphism of $\Sym^2_0 T^*M$ we have
  \begin{equation}
    \label{eq:18}
    W_+^{f^2g}(h) = f^{-2} W_+^g(h).
  \end{equation}
The proposition follows.
\end{proof}

\begin{rem}
  Delay \cite{Del24} proved the conformal invariance of $L-\frac23 \delta_0^*\delta+\frac{\Scal}{12}$. Here we proved the conformal invariance of $T^*T-W_+$. A priori we have seen that the two operators coincide for Einstein metrics, see equation (\ref{eq:11}). It turns out that the coincidence is more general, but we do not need it here.
\end{rem}

{\begin{rem}
    The same considerations also lead to a different conformal invariance for the divergence:
    \begin{equation}
      \label{eq:30}
      \delta^{f^2g} = f^{-4}\circ \delta^g \circ f^2.
    \end{equation}
\end{rem}
}

\section{Instability of compact conformally Kähler metrics}\label{sec:instability-compact}

The Hilbert-Einstein functional is defined as $\mathcal{S}:g\mapsto \int_M \Scal^gdv_g$, and its first variation for a compactly supported deformation of the metric $k$ is $-\int_M\langle E_g,k \rangle_gdv_g$, where $E_g= \Ric-\frac{\Scal^g}{2}g$. It is the second variation of this scalar curvature that determines the stability of an Einstein metric. From \cite[Proposition 4.55]{Bes87}, for a \textit{compactly supported} and \textit{traceless} {and divergence-free} deformation $k$, this second variation at an Einstein metric is
\begin{equation}\label{second variation HE}
    \mathcal{S}^{(2)}_g(k,k) = -\int_M \big\langle L_gk\,,\,k \big\rangle_gdv_g.
\end{equation}

\begin{defn}
 An Einstein metric $g$ is \textit{semistable} if the second variation of $\mathcal{S}$ at $g$ is nonpositive for \textit{all} traceless {divergence free} deformations. It is \textit{unstable} otherwise.
\end{defn}

We will apply this definition to the following setting. Let us consider $g$ an Einstein metric that is conformal to a Kähler metric $\Tilde{g}$ through the function $g=f^2\Tilde{g}$. The conformal factor is defined up to a multiplicative constant, and from \cite[Prop. 5]{Der83} we can take
\begin{equation}
  \label{eq:19}
  f^{-1}=\frac{\Scal^{\tilde g}}6.
\end{equation}
By \cite{Der83} this function cannot vanish. Our convention differs from the convention in \cite{Der83} by the factor $\frac16$: the choice (\ref{eq:19}) will be more convenient for us because $\frac16 \Scal^{\tilde g}$ is the eigenvalue of $W_+^{\tilde g}$ on the Kähler form $\tilde{\omega}_+$ of $\Tilde{g}$, and therefore from (\ref{eq:18}) the eigenvalue of $W_+^g$ on $\tilde \omega_+$ is $f^{-3}$.

Note that in (\ref{eq:19}) one should a priori take for $f^{-1}$ the absolute value of the RHS, but it turns out that both in the case of a compact Einstein $(M,g)$ with $\Scal^g>0$ and in the noncompact case with $\Scal^g=0$ one has necessarily $\Scal^{\tilde g}>0$: indeed we write the Yamabe equation
\[ \Scal^g = f^{-3} (6\Delta^{\Tilde g} + \Scal^{\Tilde g})f. \]
In the compact case $f$ has a minimum; in the noncompact case, we will see that $f$ is proper and therefore also has a minimum. At a minimum $x_0$ of $f$ it follows from the maximum principle that $\Scal^{\Tilde g}(x_0)\geq0$, so since $\Scal^{\Tilde g}$ cannot vanish, it must be positive everywhere.

Suppose that we also have a closed anti-selfdual 2-form $\omega_-$.
Then we define (for the metric $g$):
\begin{equation}\label{eq def h direction instable}
  h := \omega_-\,\circ\,f^{3}\tilde{\omega}_+ .
\end{equation}

\begin{thm}\label{thm:inst1}
  Suppose $(M^4,g)$ is compact, Einstein, conformally Kähler, with nonnegative scalar curvature. Then for any non-zero closed anti-selfdual form $\omega_-$ {the projection $h_0$ on divergence free tensors of} the tensor $h$ defined by (\ref{eq def h direction instable}) destabilizes $g$, that is
  \begin{equation}
    \label{eq:20}
    \mathcal{S}^{(2)}_g(h_0,h_0) > 0.
  \end{equation}
\end{thm}

The version of Theorem \ref{thm:inst1} for $(M,g)$ complete noncompact will be stated in Section \ref{sec:instability-alf}.

\begin{exa}
  A trivial example is that of a Kähler-Einstein Fano manifold $(M^4,g)$: in that case the instability is well-known, see formula \cite[(12.92')]{Bes87}. See also \cite{CaoHamIlm04,HalMur11} for similar results in the context of Ricci flow. 
\end{exa}
\begin{exa}
  The 4-dimensional, non Kähler but conformally Kähler, Einstein metrics were classified by LeBrun \cite{LeB12}: one gets only the Page metric on $\mathbb{C}P^2 \# \overline{\mathbb{C}P^2}$ and the Chen-LeBrun-Weber metric \cite{CheLebWeb08} on $\mathbb{C}P^2 \# 2\overline{\mathbb{C}P^2}$. So the theorem says that these two metrics are unstable for the Ricci flow and provides destabilizing tensors. The instability of the Page metric was proven, and that of the Chen-LeBrun-Weber metric was conjectured in \cite{HHS14}. Further numerical evidence of the instability of the Chen-Lebrun-Weber metric and an easier proof of the instability of the Page metric as a Ricci shrinker were provided in \cite{HalMur14,HalMur17} studying conformal deformations.
\end{exa}
\begin{rem}
   The 2-form $\tau=f^{3}\tilde{\omega}_+$ has a nice conformal interpretation: it is a Killing 2-form for $g$, that is it satisfies the equation
\begin{equation}\label{eq:21}
 \mathcal{T}^g(\tau)_X:=\nabla_X\tau - \frac13 ( X \lrcorner d\tau - X \wedge \delta\tau ) = 0.
\end{equation}
This is a conformally invariant equation: $f^{-3}\mathcal{T}^{f^2g}(f^3\tau)=\mathcal{T}^g(\tau)$ and therefore the solution $f^3\tilde \omega_+$ for $g$ corresponds to the trivial solution $\tilde \omega_+$ for $\tilde g$.
 \end{rem}

\begin{proof}
  We begin by considering $\tilde h=\omega_-  \,\tilde\circ \,\tilde \omega_+$ for the Kähler metric $\tilde g$: since $\tilde \omega_+$ is parallel and $d\omega_-=0$, we have $d_-^* \tilde h=0$ and therefore {$\delta^{\tilde g}\tilde h=0$ and}
  \begin{equation}
    \label{eq:22}
    P^{\tilde g}\tilde h = - \omega_- \,  \tilde\circ \, W_+^{\tilde g}(\tilde \omega_+)  
    = - \frac{\Scal^{\tilde g}}6 \tilde h = - f^{-1} \tilde h.
  \end{equation}

  Now observe that $h=f\tilde h$ because $\tilde \circ = f^2 \circ$. By the conformal invariance (\ref{eq:13})  {and (\ref{eq:30})} we obtain
  \begin{equation}
    \label{eq:23}
    P^g h = - f^{-3} h {, \quad \delta^g(f^{-3}h)=0}.
  \end{equation}
   {From (\ref{eq:29}) it follows that $L$ and $P$ respect the decomposition $\Sym_0 T^*M=\ker \delta \oplus \im \delta^*$; but since we already have $\delta^g(P^gh)=0$, one gets $P^gh_0=P^gh=-f^{-3}h$.}
From (\ref{second variation HE}) and (\ref{eq:11}) we have
  \begin{align}
    \mathcal{S}^{(2)}_g(h_0,h_0) & {= - \int_M \big\langle Lh_0,h_0 \big\rangle dv_g} \\
    & {= - \int_M \big\langle Lh_0,h \big\rangle dv_g} \\ 
    &= - \int_M \big\langle Ph_0 - \frac{\Scal^g}{12}h_0 , h \big\rangle dv_g \\
 & { = \int_M \big\langle f^{-3}h  + \frac{\Scal^g}{12}h,h  \big\rangle dv_g } \\
    &>0.
  \end{align}
  The instability follows.
\end{proof}

\section{Instability of conformally Kähler ALF Ricci-flat metrics}\label{sec:instability-alf}

We now extend Theorem \ref{thm:inst1} to noncompact complete Ricci flat metrics with quadratic decay of curvature, that is to \textit{gravitational instantons}. The known non Kähler examples are all conformally Kähler, and ALF in the following sense \cite{BiqGau23}:

\begin{defn}
    A Riemannian manifold $(M^{4}, g)$ is ALF if
\begin{itemize}
  \item it has one end diffeomorphic to $(A,+\infty) \times L$, where $L$ is $\mathbb{S}^{1} \times \mathbb{S}^{2}, \mathbb{S}^{3}$, or a finite quotient;
  \item there is a triple $(\eta, T, \gamma)$ defined on $L$, where $\eta$ is a 1-form, $T$ a vector field such that $\eta(T)=1$ and $T\lrcorner d \eta=0$, and $\gamma$ is a $T$-invariant metric on the distribution $\ker \eta$;
  \item the transverse metric $\gamma$ has constant curvature $+1$ ;
  \item the metric $g$ has the behaviour
    \begin{equation}
      \label{eq:24}
      g=d r^{2}+r^{2} \gamma+\eta^{2}+h, \text { with }|\nabla^{k} h|=\mathrm{O}\left(r^{-1-k}\right),
    \end{equation}
where $\gamma$ is extended to the whole $TL$ by deciding that $\ker\gamma$ is generated by $T$, and the covariant derivative $\nabla$ and the norm are with respect to the asymptotic model $d r^{2}+r^{2} \gamma+\eta^{2}$.
\end{itemize}
\end{defn}

The definition covers all the known examples: the Riemannian Kerr metrics (a special case being the Schwarzschild metric), the Taub-bolt metric, the Chen-Teo metrics \cite{CheTeo11}. In addition, the Kerr and Chen-Teo metrics are AF, that is $L=\mathbb{S}^{2} \times \mathbb{S}^{1}$, so that the end has the topology of $\left(\mathbb{R}^{3} \backslash \mathrm{B}(R)\right) \times \mathbb{S}^{1}$ for some large $R$. It is expected that this list gives all examples of conformally Kähler, non Kähler, ALF gravitational instantons, see \cite{AADNS}.

 {We also introduce a common notion of stability for ALF Ricci-flat metrics. 
\begin{defn}
     An Ricci-flat ALF metric $g$ is \textit{semistable} if the second variation of $\mathcal{S}$ at $g$ is nonpositive for \textit{all} traceless  {divergence free} deformations $h$ so that $|\nabla^kh| = O(r^{-1-k})$. It is \textit{unstable} otherwise.
\end{defn}
\begin{rem}
    The decay assumption $|\nabla^kh| = O(r^{-1-k})$ as well as the transverse-traceless assumption ensures that no boundary counter-term is required to make sense of the variations of $\mathcal{S}$ in the direction $h$. Indeed, the first variation of $\operatorname{Scal}$ in the direction $h$ vanishes pointwise while the second and higher variations are all integrable.
\end{rem}
}

To apply the technique of the previous section, we need to verify that our $2$-tensor $h=\omega_-\circ\tau$ (with $\tau$ the Killing 2-form and $\omega_-$ harmonic) decays suitably at infinity. Our gravitational instanton is conformally Kähler with $g=f^2\tilde g$. From the definition of an ALF metric we have $|W_+|=O(r^{-3})$, and actually for all the examples we have exactly $|W_+| \sim \mathrm{cst.}\, r^{-3}$. Therefore the conformal factor $f$ satisfies $f \sim \mathrm{cst.}\, r$, and the Killing 2-form $\tau=f^3 \tilde \omega_+$ satisfies
\begin{equation}
  \label{eq:25}
  |\tau| = O(r).
\end{equation}
\begin{rem}\label{rem:growth}
  This growth is true for all our examples (which are supposed to give a complete list). It is proved in \cite{BiqGau23} to be true in general in the toric case. Actually one can prove the growth (\ref{eq:25}) in the general case, which is the basic ingredient needed to make all our arguments abstract, and not specific to the examples. Since it looks likely that these known examples already exhaust all conformally Kähler, non Kähler, gravitational instantons (this seems to be proved in the recently posted article \cite{LiM23a}, as mentioned in the introduction), we avoided the proof of this general statement, and therefore the arguments below work only for the known examples. 
\end{rem}

There is a concrete construction of an anti-selfdual harmonic form $\omega_-$. Denote by $J$ the complex structure. The Kähler metric $\tilde g$ is extremal \cite{Der83}, which by definition means that the vector field $J \nabla^{\tilde{g}}\mathrm{Scal}^{\tilde{g}}$ is a Killing vector field. By (\ref{eq:19}) we rewrite (up to a constant) this vector field in terms of the metric $g$ and obtain that
\begin{equation}
  \label{eq:26}
  X = J \nabla^g f
\end{equation}
is a Killing vector field. Denoting by $\alpha$ the 1-form associated to $X$ via the metric $g$, Kostant's formula (\ref{eq:28}) below implies that $\nabla^*\nabla\alpha=\Ric\alpha=0$, therefore from the Bochner formula $\alpha$ is harmonic. Since $X$ is Killing we have $d^*\alpha=0$ and therefore $d^*d\alpha=0$. On $\Omega_\pm$ we have $d^*=\mp*d$, therefore $d_+^*d_+-d_-^*d_-=-*(dd_++dd_-)=-*d^2=0$, so we recover the well-known identity $d_+^*d_+=d_-^*d_-=\frac12 d^*d$. It follows that $d_-^*d_-\alpha=0$, therefore we obtain a closed (and co-closed) anti-selfdual form defined by
\begin{equation}
  \label{eq:27}
  \omega_-=d_-\alpha .
\end{equation}
The asymptotics of the Killing form $\tau$ determined in \cite{BiqGau23} in the toric case (this contains all our examples, see Remark \ref{rem:growth}) implies that, up to a constant, one has $\alpha \sim \eta$, where $\eta$ is the 1-form from the definition (\ref{eq:24}). Since $T\lrcorner d\eta=0$, one has $|d\eta|=O(r^{-2})$ and therefore $|\omega_-| = O(r^{-2})$. So we obtain:


\begin{prop}    \label{prop decay 2 forms}
  For the Kerr, Taub-bolt and Chen-Teo metrics, the tensor $h=\omega_-\tau$, where $\tau$ is the Killing 2-form coming from the conformally Kähler structure and $\omega_-$ is the harmonic anti-selfdual form defined by (\ref{eq:27}), is nonzero and satisfies, for $l\in\mathbb{N}$ and some $C_l>0$, the bounds
    \begin{equation}\label{eq: decay h}
        |\nabla_g^l h|\leqslant C_l r^{-1-l}. 
    \end{equation}
\end{prop}
\begin{proof}
  We already proved in these cases the $C^0$ bound and the bounds on the derivatives are obtained in the same way. There remains to prove that $h$ is nontrivial, that is $\omega_-$ is nontrivial.

  The proof does not refer to any specific metric: suppose $\omega_-=0$, then we will prove that $g$ is hyperKähler for the opposite orientation. Indeed, if $\omega_-=0$ then $A:=\nabla\alpha$ is a selfdual 2-form (actually, up to a constant, it is $\tilde \omega_+$ but we do not need this fact).

  For the Killing field $X$ in \eqref{eq:26} we use Kostant's formula: for a tangent vector $Y$,
  \begin{equation}
    \label{eq:28}
    \nabla_YA=-R_{X,Y}.
  \end{equation}
  Since $A$ is a selfdual 2-form, we obtain that $R_{X,Y}$ is selfdual for all tangent vectors $Y$. Since $\Ric(g)=0$ the curvature operator $R$ preserves $\Lambda_\pm$, so for any $Y$ we have $R_{X\wedge Y+*(X\wedge Y)}\in \Lambda_+$ and therefore $R_{*(X\wedge Y)}\in \Lambda_+$. Since the collection of $\{ X \wedge Y, *(X \wedge Y)\}$ when $Y\in TM$ generates all 2-forms,  the image of the curvature operator is entirely contained in $\Lambda_+$. Since $R$ acts on $\Lambda_-$ by $W_-$ we deduce that $W_-=0$, that is $g$ is hyperKähler for the other orientation, which is contrary to the hypothesis that $g$ is non Kähler.
\end{proof}

\begin{rem}\label{remark Taub NUT EH}
  Conversely, in the case of a conformally Kähler metric which is hyperKähler for the other orientation, that is satisfying $W_-=0$ (Eguchi-Hanson, Taub-NUT), the form $\omega_-$ is trivial: one can reverse the above argument, or more directly observe that the Weitzenböck formula
  \[ d_-d_-^*\omega_- = \frac12 \nabla^*\nabla\omega_- - W_-(\omega_-) \]
implies $\nabla^*\nabla\omega_-=0$ if $W_-=0$. Since $|\omega_-|=O(r^{-2})$ one can integrate by part to obtain $\nabla\omega_-=0$. This is consistent with the fact that hyperKähler metrics are stable.
\end{rem}

\begin{thm}
  The Kerr, Taub-Bolt and Chen-Teo metrics are unstable.
\end{thm}
Remind that this is conjectured to be the full list of conformally Kähler, non Kähler, gravitational instantons.
\begin{proof}
  Let $(M,g)$ be one of these gravitational instantons.
  Consider the 2-tensor $h$ coming from Proposition \ref{prop decay 2 forms}. The situation is the same as in the proof of Theorem \ref{thm:inst1}, so we have $Ph=-f^{-3}h$.

   {The divergence free part $h_0$ of $h$ is obtained as $h_0=h-\delta_0^*\alpha$, where $\delta\delta_0^*\alpha=\delta h$. To construct $\alpha$ one can use the b-calculus analysis (see \cite{HauHunMaz04}): a priori \cite{HauHunMaz04} can be applied only in the case where the Killing field $X$ has closed orbits, but by the toric invariance one can reduce to toric invariant objects and then apply the results to all cases.

    It follows from the weighted analysis of the operator $\delta\delta_0^*$ that $h_0$ satisfies the same bounds (\ref{eq: decay h}) as $h$. Then we have $Lh_0=Ph_0=-f^{-3}h$, and since $h=O(r^{-1})$ and $h_0=O(r^{-1})$,}
\[ \int_M - \langle Lh_0,h_0 \rangle dv_g = \int_M - \langle Lh_0, h \rangle dv_g = \int_M f^{-3} |h|^2 dv_g < +\infty. \]
With the bounds (\ref{eq: decay h}) on the derivatives of $h$, a standard cutoff argument gives a compactly supported $k$ with $\int_M -\langle Lk,k \rangle dv_g > 0$, which proves instability.
\end{proof}

\begin{exa}
    On the (Euclidean) Schwarzschild metric:
    $ g =  \frac{dr^2}{1-\frac{2m}{r}}+\big(1-\frac{2m}{r}\big)dt^2+ r^2 g_{\mathbb{S}^2},$  conformal to the Kähler metric $\tilde{g} = r^{-2}g$, one finds $$h = - \frac{1}{r}\Big(\frac{dr^2}{1-\frac{2m}{r}}+\Big(1-\frac{2m}{r}\Big)dt^2 - r^2 g_{\mathbb{S}^2}\Big).$$
\end{exa}

\bibliographystyle{abbrv}
\bibliography{biblio,biquard,ickem}

\end{document}